\documentclass[11pt,letterpaper]{article}

\usepackage[letterpaper,hmargin=1.25in,vmargin=1.5in]{geometry}
\usepackage{natbib}
\usepackage{amssymb, latexsym, amsmath, amsthm, mathrsfs}
\usepackage{amsbsy}
\usepackage{graphicx}
\usepackage[usenames]{color}
\usepackage{ifthen}
\usepackage{soul}
\usepackage{setspace}
\usepackage{hyperref}
\theoremstyle{plain}
\newtheorem{theorem}{Theorem}[section]
\newtheorem{corollary}{Corollary}[section]
\newtheorem{lemma}{Lemma}[section]
\newtheorem{proposition}{Proposition}[section]
\theoremstyle{definition}

\newtheorem{assumption}{Assumption}[section]
\theoremstyle{remark}

\renewcommand{\phi}{\varphi}

\begin{document}

\title{Posterior consistency in misspecified models for i.n.i.d response}
\author{ {\sc Karthik Sriram}\footnote{Karthik Sriram  is Assistant Professor, Production and Quantitative Methods Area, Indian Institute of Management Ahmedabad, India ({\tt karthiks@iimahd.ernet.in})} \ \ and \;{\sc R.~V.~Ramamoorthi}\footnote{R.~V.~Ramamoorthi is Professor, Department of Statistics and Probability, Michigan State University, USA ({\tt ramamoor@stt.msu.edu})}   }


\date{}
\maketitle
\thispagestyle{empty}

\doublespacing

\begin{abstract}
 We derive conditions for posterior consistency when the responses are independent but not identically distributed ($i.n.i.d$) and the model is  ``misspecified" to be a family of densities   parametrized by a possibly infinite dimensional parameter.  Our approach has connections to key ideas developed  for $i.i.d$ models in \citet{Kleijn_van2006} and it's subsequent simplification in \cite{rvr.sriram.martin} (unpublished manuscript). While key results in these two papers rely heavily on the convexity of the specified family of densities, parametric families  are seldom convex. In this note, we take a  direct approach to deriving posterior consistency with respect to natural topologies on the parameter space without having to impose conditions on the convex hull of the parametric family. We first derive our results for the case when the responses are $i.i.d$ and then extend it to the $i.n.i.d$ case. As an example, we demonstrate the applicability of the results to the Bayesian quantile estimation problem. 
\smallskip

\end{abstract}

\section{Introduction}
In many applications, a statistical model for a random response variable $Y_i$ is specified in the form of a density $f_{\theta(X_i)}(\cdot)$, where the parameter $\theta(\cdot)$ is a function on the space of possible values of the covariate vector ${\bf X}_i$. A common example is the ordinary least squares linear regression where the response is assumed to be normally distributed with mean being an unknown linear function of the covariates. Such a model specification is meant to be a simplified representation of a more complex reality. Arguably, the model is therefore a misspecification of the true underlying probability distribution. However, knowledge of the entire probability distribution is often less relevant and of more interest may be the specific parameter $\theta(\cdot)$. For example, $\theta({\bf X}_i)$ could be a particular quantile of $Y_i$ given the covariates ${\bf X}_i$.  Therefore, even if the model is misspecified, it is desirable that the statistical method leads to correct inference on the parameters of interest. Bayesian analysis of such models proceeds by endowing a prior probability distribution $\Pi(\cdot)$ on the parameter space of $\theta(\cdot)$. The prior $\Pi(\theta)$ along with the specified model $f_{\theta(\cdot)}$, leads to the  conditional posterior distribution of $\theta|(Y_1,Y_2, \cdots, Y_n)$. Then, a desirable asymptotic property is posterior consistency, which requires that the posterior distribution concentrate around the true parameter $\theta_0(\cdot)$ as the sample size $n$ increases. 

In this note, we derive conditions for posterior consistency  when the responses are independent but not identically distributed ($i.n.i.d$) with $Y_i\sim $ (true density) $p_{0i}$, whereas the model for $Y_i$ is  ``misspecified" to be  a family of densities  $\{f_{\theta(X_i)}, \ \theta\in \Theta\}$, parametrized by a possibly infinite dimensional parameter $\theta(\cdot)$. 

Past studies on posterior consistency under misspecification have mostly focused on i.i.d models. Following the early work of \cite{Berk1966}, an exhaustive study of parametric models is carried out by  \citet{Bunke_milhaud1998, Lee} and in the nonparametric set up by \citet{Kleijn_van2006, DW2013}. \cite{Kleijn_van2012} give a Bernstein-von-Mises theorem for misspecified models.  In a recent yet unpublished manuscript \cite{rvr.sriram.martin} provide a simple proof of the consistency result in \citet{Kleijn_van2006} for non-parametric convex models and  also derive some specific results for weak as well as $L_1$ consistency. The study of non i.i.d misspecified models is relatively limited. \cite{shalizi2009} considers the infinite dimensional non-parametric case in the non i.i.d set up deriving general results but under somewhat stringent conditions. \citet{sriram.rvr.ghosh.2013} derive posterior consistency for i.n.i.d response in the specific case of Bayesian linear quantile regression model based on the asymmetric Laplace model specification. 

Our approach to the i.n.i.d case in this paper has connections to key ideas developed in \citet{Kleijn_van2006} for i.i.d models and it's subsequent simplification in \cite{rvr.sriram.martin}. Key results in these papers rely heavily on the convexity of the specified family of densities. However, parametric families  are seldom convex and it is desirable to have a more direct approach to deriving posterior consistency with respect to (w.r.t) natural topologies on $\theta(\cdot)$ without having to impose conditions on the convex hull of the parametric family. We circumvent the convexity requirement by making a continuity assumption on expected likelihood ratio. Apart from this,  as is standard for misspecified models,  we require that there be a parameter value $\theta^*$ such that $f_{\theta^*}$ minimizes the Kullback-Leibler (KL) divergence from the true density and that the prior put sufficient probability mass around it. 

In what follows, we first develop our approach for the case of i.i.d response in Section \ref{seciid}.  In Section \ref{secinid}, we extend the approach to the i.n.i.d case, deferring the details of the proof to Section \ref{proofinid}. We demonstrate the applicability of results using the example of Bayesian quantile estimation, both in the i.i.d as well as the i.n.i.d case. 
\section{The i.i.d. case}
\label{seciid}
Let $Y_1, Y_2, \cdots , Y_n$ be $i.i.d.$ from some density $p_0$. We will denote by $P_0$ the product measure $p_0^{\infty}$ and by $ E[\cdot]$ the expectation w.r.t this product measure. Suppose the model is specified as a family $\mathcal{F}=\left\{f_{\theta}: \ \theta \in \Theta \right\}$, which may not contain the true density $p_0$.  Let $\Pi(\cdot)$ be a prior on the parameter space $\Theta$. Our interest is in the asymptotic concentration of the posterior obtained using the model $f_{\theta}$ and the prior $\Pi$. Let $d(\cdot,\cdot)$ be a metric on $\Theta$.  It is  known that posterior concentrates around a parameter value $\theta^*$ corresponding to a density $f_{\theta^*}$ that minimizes the Kullback-Leibler divergence from $p_0$.  The posterior probability of the set $U^{c}=\{\theta: \ d(\theta, \theta^*)>\epsilon\}$ can be written as follows:
\begin{eqnarray*}
\Pi_n(U^{c}):=\frac{\int_{U^{c}} \prod_{i=1}^{n}\frac{f_{\theta}(y_i)}{f_{\theta^*}(y_i)} d\Pi(\theta)}{\int_{\Theta} \prod_{i=1}^{n}\frac{f_{\theta}(y_i)}{f_{\theta^*}(y_i)} d\Pi(\theta)}=:\frac{R_{1n}}{R_{2n}}
\end{eqnarray*}
Our interest is in deriving conditions under which this posterior probability goes to zero. This can be accomplished by first showing the denominator $R_{2n}$ tends to $\infty$ at a certain rate and then by suitably bounding the numerator $R_{1n}$. A natural condition to handle the denominator  is that the prior puts a positive mass on Kullback-Leibler neighborhoods. The following two assumptions and proposition help handle the denominator.
\begin{assumption}
\label{existencepos}
$\exists $  $\theta^* \in \Theta $ such that $\theta^*=arg\min_{\theta\in \Theta} E\log\frac{p_0}{f_{\theta}}$ and $\theta^*$ is in the $d-$ support of $\Pi$.
\end{assumption}
\begin{assumption}
 \label{cnty}
$E\log\frac{f_{\theta}}{f_{\theta^*}}$  is continuous in $\theta$  and for any $\theta_1 \in \Theta$, $E\left[\frac{f_{\theta}}{f_{\theta_1}}\right]$ is continuous in $\theta$, w.r.t the metric $d$.
  \end{assumption}
The proposition below is standard and  is used to establish the limiting property of the denominator of the posterior probability. 
\begin{proposition}
\label{denom}
Under assumptions \ref{existencepos} and \ref{cnty},  
\[\mbox{ for any } \beta>0, \ e^{n\beta}R_{2n} \rightarrow \  \infty \ a.s. \ P_0.\]
\begin{proof} The two assumptions \ref{existencepos} and \ref{cnty} ensure that KL neighborhoods of $\theta^*$ of the form $\{\theta: \ E\log\frac{f_{\theta^*}}{f_{\theta}}< \epsilon \}$ will get  a positive mass under the prior. Rest of the proof is similar to Lemma 4.4.1 of \cite{Ghosh_RVR2003}. \end{proof}
\end{proposition}

The next assumption relates the metric $d$  and the Kullback-Leibler divergence.
\begin{assumption}
\label{KLind} 
Every $d$-neighborhood  of $\theta^*$:   $U=\left\{\theta : d(\theta, \theta^*)<\epsilon \right\} $
contains a Kullback-Leibler neighborhood of the form
$\left\{\theta\in \Theta : \ E\log\frac{f_{\theta^*}}{f_{\theta}}<\delta \right\}$ for some $\delta \in (0,1)$.
\end{assumption}
A property that helps bound the numerator of the posterior probability is a condition introduced in \cite{Kleijn_van2006}. If $\mathcal{F}$ is the family of functions with $f^*$ being the KL minimizer, this condition requires that for a convex set $A$ and some $\delta>0$, we have $\sup_{f\in A}\inf_{\alpha \in [0,1]} E\left(\frac{f}{f^{*}} \right)^{\alpha} < e^{- \delta}$. 
As shown in \cite{rvr.sriram.martin} the condition implies that for a convex set $A\subset U^{c}$, $\exists \ \alpha'$ such that   $E   \left(\int_{A} \prod_{i=1}^{n}\frac{f(Y_{i})}{f^*(Y_i)}  d\Pi(f)\right)^{\alpha'}< e^{-n\delta}$.  Convexity of the set $A$ is crucial in these arguments. In the semi-parametric set up, we may not have convexity for any sub-collection of the class of densities.  As shown in the following proposition, assumptions \ref{cnty} and \ref{KLind} help circumvent the convexity requirement.
\begin{proposition}
\label{numer}
Suppose assumptions  \ref{existencepos}, \ref{cnty} and \ref{KLind}  hold.    Then for any $\theta_1\in U^{c}, \ \exists$ an open set $A_{\theta_1}$ containing $\theta_1$ such that for some $\delta>0$, $\alpha \in (0,1)$ and for any probability measure $\nu(\cdot)$ on $A_{\theta_1}$, we have:
\begin{itemize}
\item[{(a)}] $E\left[ \left(\int_{A_{\theta_1}}\frac{f_{\theta}(y)}{f_{\theta^*}(y)} d\nu(\theta) \right)^{\alpha}\right]< e^{-\alpha\frac{\delta}{2}}$.
\item[{(b)}]$E\left[ \left(\int_{A_{\theta_1}}\prod_{i=1}^{n}\frac{f_{\theta}(y_i)}{f_{\theta^*}(y_i)} d\nu(\theta) \right)^{\alpha}\right]< e^{-n\alpha\frac{\delta}{2}}$. 
\end{itemize}
\end{proposition}
We defer the proof of this proposition to section \ref{proofinid} and now present the posterior consistency result for the i.i.d case when the parameter space is compact.
\begin{theorem}
\label{thmiid}
If $\Theta $ is compact and assumptions \ref{existencepos},  \ref{cnty} and \ref{KLind} hold, then  $\Pi_{n}(U^{c}) \rightarrow \ 0 \ a.s. \ P_0$.
\begin{proof} Note that proposition  \ref{numer} can also be applied by taking $ \nu(\cdot)=\frac{\Pi(\cdot)}{\Pi(A_{\theta_1})}$. So, for any $\theta_1\in U^c$, $\exists \ \delta>0$ and an open set $A_{\theta_1}$ containing $\theta_1$ such that ,
\begin{eqnarray*}
&&P\left(  \left(e^{n\frac{ \delta}{4}}\int_{A_{\theta_1}} \prod_{i=1}^{n} \frac{f_{\theta}(y_i)}{f_{\theta^*}(y_i)}d\Pi(\theta)\right)^{\alpha}> \epsilon^{\alpha} \right)\\
&&\leq \frac{\Pi^{\alpha}(A_{\theta_1})}{\epsilon^{\alpha}}\cdot E\left(e^{n\frac{ \delta}{4}}\int_{A_{\theta_1}} \prod_{i=1}^{n} \frac{f_{\theta}(y_i)}{f_{\theta^*}(y_i)}\frac{d\Pi(\theta)}{\Pi(A_{\theta_1})}\right)^{\alpha}\leq \frac{e^{-n\alpha \frac{\delta}{4}}}{\epsilon^{\alpha}}.
\end{eqnarray*}
Therefore, by Borel-Cantelli Lemma, we can conclude that
\[e^{n\frac{ \delta}{4}}\int_{A_{\theta_1}} \prod_{i=1}^{n} \frac{f_{\theta}(y_i)}{f_{\theta^*}(y_i)}d\Pi(\theta) \rightarrow \ 0 \ a.s. \ P_0. \]
By proposition \ref{denom}, it follows in particular that 
\[e^{n\frac{ \delta}{4}}\int_{\Theta} \prod_{i=1}^{n} \frac{f_{\theta}(y_i)}{f_{\theta^*}(y_i)}d\Pi(\theta) \rightarrow \ \infty \ a.s. \ P_0. \]
Considering the ratio of the above two quantities immediately gives $\Pi_n(A_{\theta_1})\rightarrow \ 0 \ a.s. \ P_0$.
By compactness, $U^c$ can be covered by finitely many sets of the form $A_{\theta_1}$. Hence the result follows.
\end{proof}
\end{theorem}  
\begin{corollary}
\label{coriid}
Suppose  assumptions \ref{existencepos},  \ref{cnty} and \ref{KLind} hold. If the parameter space $\Theta$ can be written as  $\Theta_1 \cup \Theta_2$ such that $\Theta_1$ is compact and  for some  $\delta_1>0 $ we have  $E\left(\frac{f_{\theta}(Y)}{f_{\theta^*}(Y)}\right)<e^{-\delta_1}$, then  $\Pi_{n}(U^{c}) \rightarrow \ 0 \ a.s. \ P_0$.
\begin{proof}
First, it follows from theorem \ref{thmiid} that $\Pi_{n}(G\cap \Theta_1) \rightarrow \ 0 \ a.s. \ P_0$. In order to show $\Pi_{n}(G\cap \Theta_2) \rightarrow \ 0$, note that
\begin{eqnarray*}
E\left[ \int_{\Theta_2}\prod_{i=1}^{n}\frac{f_{\theta}(y_i)}{f_{\theta^*}(y_i)} d\Pi(\theta) \right]= \int_{\Theta_2}\prod_{i=1}^{n} E\left[\frac{f_{\theta}(y_i)}{f_{\theta^*}(y_i)}\right] d\Pi(\theta)\leq e^{-n\delta_1}.
\end{eqnarray*}
Now, using similar arguments as in theorem \ref{thmiid} it follows that $\Pi_{n}(\Theta_2) \rightarrow \ 0 \ a.s. \ P_0$.
\end{proof}
\end{corollary} 
  \subsection{Example: Bayesian quantile estimation.}
  \label{eg1}
  Consider a family of asymmetric Laplace densities (ALD), i.e. $\mathcal{F} =\{f_{\theta}: \ \theta\in \Theta \}$, where $\Theta \subseteq (-\infty, \infty)$,  $f_{\theta}(y)=\tau(1-\tau) e^{-(y-\theta)(\tau-I_{y\leq \theta})}$ for $y \in (-\infty, \infty)$. The parameter $\theta$ can be interpreted as the $\tau^{th}$ quantile of the density $f_{\theta}$. Let $P_0$ be the true underlying distribution of $Y$ with a unique $\tau^{th}$ quantile given by $\theta_0$. We are interested in the posterior probability of the set $\{|\theta - \theta_0|>\epsilon\}$.
  
We will first consider the case when $\Theta$ is compact. By the properties of ALD, it can be seen (see proposition 1, lemmas 1 and 2 of \citealt{sriram.rvr.ghosh.2013}) that (a) $\theta^*=\theta_0$, (b) $|\log\frac{f_{\theta^*}}{f_{\theta}}|\leq |\theta-\theta^*|$ and (c) that if $|\theta-\theta^*|> \epsilon$ then:
 \[ E\log\frac{f_{\theta^*}}{f_{\theta}}> \delta = \frac{\epsilon}{2}\cdot \min\left\{P_0\left(0<Y-\theta_0<\frac{\epsilon}{2}\right), P_0\left(-\frac{\epsilon}{2}<Y-\theta_0<0\right) \right\}. \]
  
 It follows by (b) that $E\log\frac{f_{\theta^*}}{f_{\theta}}$ is continuous and further  by compactness of $\Theta$ that $\frac{f_{\theta}}{f_{\theta_1}}$ is uniformly bounded. An application of dominated convergence theorem(DCT) would then imply that $E\left[ \frac{f_{\theta}}{f_{\theta_1}}\right]$ is continuous. Hence assumption \ref{cnty} is satisfied and \ref{existencepos} is satisfied as long as the prior puts positive mass on all neighborhoods of $\theta_0$. Finally, using (c), as long as the density function of $P_0$ is positive and continuous at the true quantile $\theta_0$, assumption \ref{KLind} is satisfied. So, theorem \ref{thmiid} applies.
  
To deal with the case when $\Theta=(-\infty, \infty)$, we will use corollary \ref{coriid}. Lemma 1 of \citealt{sriram.rvr.ghosh.2013}) gives the following useful inequality.
\begin{eqnarray*}
\log\frac{f_{\theta}(Y_i)}{f_{\theta_0}(Y_i)}< -|\theta-\theta_0|\cdot \min\{\tau, 1-\tau\} + |Y_i-\theta_0|
\end{eqnarray*}
Suppose $E|Y-\theta_0|<\infty$, then using Strong Law of Large Numbers (S.L.L.N),  $\exists \ n_0$ such that $\forall \ n\geq n_0$, $\sum_{i=1}^{n}|Y_i-\theta_0| < 2nE|Y-\theta_0|$. Now, if  $\Theta_1=\left[-\frac{3E|Y-\theta_0|}{\min\{\tau,1-\tau\}},  \frac{3E|Y-\theta_0|}{\min\{\tau,1-\tau\}}\right]$ and $\Theta_2=\Theta_1^{c}$, then $\forall \ n\geq n_0, \ \forall \ \theta\in \Theta_2$, we would have $\sum_{i=1}^{n}\log\frac{f_{\theta}(Y_i)}{f_{\theta^*}(Y_i)}< - nE|Y-\theta_0|$ and hence we would have $E\left[\prod_{i=1}^{n}\frac{f_{\theta}(Y_i)}{f_{\theta_0}(Y_i)}\right]\leq e^{ - nE|Y-\theta_0|}$.  Since $Y_i$ are $i.i.d.$ it follows that $E\left[\frac{f_{\theta}(Y)}{f_{\theta_0}(Y)}\right]\leq e^{ - \delta_1}$ for $\delta_1=E|Y-\theta_0|$. In summary,  for the possibly misspecified ALD model $\{f_{\theta}, \ \theta\in(-\infty, \infty)\}$, the posterior  concentrates around the true quantile value, $i.e. \ \Pi_n(|\theta-\theta_0|>\epsilon)\rightarrow \ 0 \ a.s. \ P_0$, if (i) the true density $p_0$ is continuous and positive at the true quantile $\theta_0$, with finite expectation and (ii) the prior $\Pi$ puts positive mass on all neighborhoods of $\theta_0$.

\section{The i.n.i.d. case}
\label{secinid}
 We will now extend the ideas from the previous section to the i.n.i.d case. We assume that the distribution of the response $Y$ is determined in principle by the knowledge of a covariate vector ${\bf X}$. In other words, there exists an unknown ``true" density function $p_{0 {\bf x}}(\cdot)$ with ${\bf x}\in \mathcal{X}$, such that $Y|{\bf X}={\bf x}\sim p_{0{\bf x}}$.  So, for the $i^{th}$ observed response $Y_i$ with covariate value ${\bf X}_i={\bf x}_i$, $Y_i \sim p_{0 {\bf x}_i}$. The ${\bf X}_i$ could be non-random and hence $Y_1, Y_2,\cdots, Y_n$ are independent but non-identically distributed. $ E_{{\bf x}}[\cdot]$ will denote the expectation w.r.t the density $p_{0{\bf x}}$. We will denote by $P_0$ the infinite product measure $p_{0{\bf x}_1}\times p_{0{\bf x}_2}\times \cdots$ and by $E[\cdot]$, the expectation w.r.t this product measure.   

Suppose we have a family of densities $\mathcal{F}=\left\{f_{t}: \ t\in [-M,M] \right\}$. Let $\Theta$ be a class of  continuous functions from $\mathcal{X}$ to  $[-M,M]$. For ease of notation, we write $\theta({\bf x})$ as $\theta_{\bf x}$. The specified model is that  $Y_i \sim f_{\theta_{{\bf x}_i}}$, where $\theta \in \Theta$ is the unknown possibly infinite dimensional parameter. A simple example is  the simple linear regression problem where $f_{t}$ would be $N(t,1)$ and $\theta_x=\alpha + \beta x$.  

We will consider the sup-norm metric on $\Theta$ and denote it by $d(\cdot,\cdot)$ to derive the results. Let $\Pi(\cdot)$ be a prior on the parameter space $\Theta$. Extending the ideas developed in the i.i.d case, we can obtain the analogous assumptions for the i.n.i.d case. The following three assumptions are analogous to assumptions \ref{existencepos}, \ref{cnty} and \ref{KLind}.
\begin{assumption}
\label{existenceposinid}
$\exists $  $\theta^* \in \Theta $ such that $\theta^*_{{\bf x}}=arg\min_{t\in [-M,M]} E_{{\bf x}}\log\frac{p_{0 {\bf x}}}{f_{t}}, \ \forall \ {\bf x} \ \in \mathcal{X}$ and $\theta^*$ is in the sup-norm support of $\Pi$.
\end{assumption} 

\begin{assumption}
\label{cntyinid}
\begin{itemize}
\item[]
\item[] a) $E_{{\bf x}}\left[ \log\frac{f_{t}}{f_{t'}}\right]$ and $E_{{\bf x}}\left[ \left(\frac{f_{t}}{f_{t'}}\right)^{\alpha}\right]$  for every $\alpha\in[0,1]$, are continuous functions in $({\bf x},t,t')\in \mathcal{X}\times [-M, M]^2$.
\item[] b) $E_{\bf x}\log^2\frac{f_t}{f_{t'}}$ is bounded for $({\bf x}, t,t') \in \mathcal{X}\times [-M, M]^2$.
\end{itemize}
\end{assumption}
 Condition (a) in the above assumption will hold if $\frac{f_t(y)}{f_t'(y)}$ is continuous in $(t,t')$ for each $y$ and if $p_{{\bf x}}(y)$ can be bounded by an integrable function in $y$. Condition (b) as will be seen later is to enable the application of S.L.L.N for independent random variables.
 \begin{assumption}
\label{KLindinid}
For any $\epsilon>0$, $\exists \ \delta\in (0,1)$ such that
\[ \left\{ t \in [-M, M]: E_{{\bf x}}\log\frac{f_{\theta^*_{\bf x}}}{f_{t}} <\delta\right\} \subseteq \left\{ t: \ |t-\theta^*_{\bf x}|<\epsilon\right\}, \ \forall \ {\bf x}\in\mathcal{X}.\]
\end{assumption}
 
 We make the following assumption with regard to the covariate space and the parameter space.
  \begin{assumption}
  \label{compactinid}
  The covariate space $\mathcal{X}$ is compact w.r.t a norm $\|\cdot\|$ and $\Theta$ is  a compact subset of continuous functions from $\mathcal{X}\rightarrow \mathcal{R}$ endowed with the sup-norm metric, i.e. $d(\theta_1, \theta_2)=\sup_{x\in \mathcal{X}}|\theta_1(x) -\theta_2(x)|$.
  \end{assumption} 
  Further,  in order for the parameter $\theta$ to be estimable, we would need some kind of a condition on the spread of points in the set $\{x_i, i\geq 1\}$ w.r.t the space $\mathcal{X}$. For example, if $\theta(x)=\alpha +\beta x$, i.e. a function involving two parameters, then having all $x_i$  equal to a constant would cause identifiability issues. In that case, we would need that the $x_i$'s take at least two distinct values for infinitely many $i$'s. The following condition helps avoid such issues.

 \begin{assumption}
 \label{denseinid}
 For any given ${\bf x}_0\in \mathcal{X}$, $\delta'>0$, let $A_{{\bf x}_0,\delta'}=\{{\bf x}: \ \|{\bf x}-{\bf x}_0\|<\delta' \}$ and $I_{A_{{\bf x}_0,\delta'}}({\bf x})$  be the indicator function which is $1$ when ${\bf x}\in A_{{\bf x}_0}$ and $0$ otherwise. Then, $\kappa({\bf x}_0, \delta')= \liminf_{n\geq 1} \frac{1}{n}\sum_{i=1}^{n}I_{A_{{\bf x}_0,\delta'}}({\bf x}_i)>0.$
 \end{assumption}
 
 As before, we can write the posterior probability of a set $U^{c}= \{\theta\in \Theta : d(\theta, \theta^*)>\epsilon\}$ as follows:
\begin{eqnarray*}
\Pi_n(U^{c}):=\frac{\int_{U^{c}} \prod_{i=1}^{n}\frac{f_{\theta_{{\bf x}_i}}(Y_i)}{f_{\theta^*_{{\bf x}_i}}(Y_i)} d\Pi(\theta)}{\int_{\Theta} \prod_{i=1}^{n}\frac{f_{\theta_{{\bf x}_i}}(Y_i)}{f_{\theta^*_{{\bf x}_i}}(Y_i)} d\Pi(\theta)}=:\frac{R'_{1n}}{R'_{2n}}
\end{eqnarray*}
 
Now, similar to the i.i.d case, the following two propositions help prove the posterior consistency result for the i.n.i.d case. 
\begin{proposition}
\label{denominid}
 Under assumptions \ref{existenceposinid} and \ref{cntyinid}, 
\[ \mbox{ for any } \beta>0, \  e^{n\beta}R'_{2n} \rightarrow \ \infty \ a.s. \ P_0.\]
\end{proposition}
\begin{proposition}
 \label{numerinid}
 Suppose assumptions \ref{existenceposinid} to \ref{denseinid} hold.  Then for any $\theta'\in U^{c}, \ \exists$ an open set $A_{\theta'}$ containing $\theta'$ such that for some $\alpha \in (0,1)$, $\delta \in (0,1)$ and for any probability measure $\nu(\cdot)$ on $A_{\theta'}$, for all sufficiently large $n$, we have:
 \begin{eqnarray*}
 E\left[ \left(\int_{A_{\theta'}}\prod_{i=1}^{n}\frac{f_{\theta_{{\bf x}_i}}(Y_i)}{f_{\theta^*_{{\bf x}_i}}(Y_i)} d\nu(\theta) \right)^{\alpha}\right]< e^{-n\alpha\frac{\delta}{2}}.
 \end{eqnarray*}
 \end{proposition}
The proofs of propositions \ref{denominid} and \ref{numerinid} are discussed in section \ref{proofinid}. We now state the main theorem that gives the posterior consistency result for the i.n.i.d case.
 \begin{theorem}
 \label{thminid}
 Suppose that assumptions  \ref{existenceposinid} to \ref{denseinid} hold. Then,
 \[\Pi_n(U^c)\rightarrow \ 0 \ a.s. \ [P_0].\]
 \begin{proof} 
 Proof is similar to that of theorem \ref{thmiid} and is an immediate consequence of propositions  \ref{denominid} and \ref{numerinid}.
 \end{proof}
   \end{theorem}
 
 \subsection{Example: Bayesian nonlinear quantile regression.}
\label{eg2} 
\cite{Koenker_basset1978} introduced Quantile Regression as a way to model any particular quantile of the response variable as a function of covariates. Given the response variable $Y_{i}$ and covariate vector ${\bf X}_{i}$ ($i=1,2,...,n$), this involves solving for $\boldsymbol{\beta}$ in the following problem.
\begin{eqnarray*}
\min_{\boldsymbol{\beta}} \sum_{i=1}^{n}\rho_{\tau}(Y_{i}-{\bf X}_{i}^{T}\boldsymbol{\beta}),
\end{eqnarray*}
where $\rho_{\tau}(u)= u(\tau - I_{(u \le 0)})$ with $I_{(\cdot)}$ being the indicator function and $0<\tau<1$. This can be formulated as a maximum likelihood estimation problem by assuming asymmetric Laplace distribution (ALD) for the response, {\it i.e.} $Y_{i} \sim ALD(., \mu_{i}^{\tau},\tau)$, where
\begin{eqnarray}
&&ALD(y; \mu^{\tau},\tau) = \tau (1-\tau) exp \left\{ -(y-\mu^{\tau})(\tau-I_{(y \le \mu^{\tau})}) \right\},   \ y \in(-\infty, \infty)
\end{eqnarray}
\cite{Yu_moyeed2001} proposed the idea of Bayesian quantile regression by assuming ALD for the response.  \cite{sriram.rvr.ghosh.2013} derive posterior consistency for Bayesian linear  quantile regression parameters based on ALD. In a recent unpublished manuscript \cite{sriram.rvr.ghosh.2013b} show among other things that the posterior consistency property holds also for a non-linear quantile regression model. We derive the result for the Bayesian nonlinear quantile regression as a special case of our formulation. 
 
  Consider a family of asymmetric Laplace densities (ALD), i.e. $\mathcal{F} =\{f_{t}: t \in [-M, M] \}$, where  $f_t(y)=ALD(y,t,\tau)$. Let the ``true" quantile function of $Y$ given covariate ${\bf X}$ be $\theta_0({\bf X})$. Assumptions \ref{compactinid} and \ref{denseinid} are on the parameter space and covariate space, which we will assume to hold.  Using similar arguments as in section \ref{eg1},  it is easy to see that assumption \ref{existenceposinid} is satisfied with $\theta^*=\theta_0(\cdot)$ and that if $|t-\theta^*_{{\bf x}_i}|> \epsilon$ then,
   \[ E_{{\bf x}_i}\log\frac{f_{\theta^*_{{\bf x}_i}}}{f_{\theta}}> \delta_{{\bf x}_i} = \frac{\epsilon}{2}\cdot \min\left\{P_{0{\bf x}_i}\left(0<Y_i-\theta^*_{{\bf x}_i}<\frac{\epsilon}{2}\right), P_{0{\bf x}_i}\left(-\frac{\epsilon}{2}<Y_i-\theta^*_{{\bf x}_i}<0\right) \right\}. \]
   If  $P_{0{\bf x}}\left(0<Y-\theta^*_{{\bf x}}<\frac{\epsilon}{2}\right)$ and $P_{0{\bf x}}\left(-\frac{\epsilon}{2}<Y-\theta^*_{{\bf x}}<0\right) $ (where $Y\sim P_{0{\bf x}}$) are continuous and positive functions of ${\bf x}$, then $\{\delta_{{\bf x}_i},\ i\geq 1\}$ can be uniformly bounded below by a positive number. Hence, assumption \ref{KLindinid} is satisfied. Similarly, since $\log\frac{f_t}{f_{t'}}$ is bounded by $|t-t'|$ and $\Theta$ is compact w.r.t sup-norm,  assumption \ref{cntyinid} is satisfied if $p_{{\bf x}}(y)$ is continuous in ${\bf x}$ for each $y$ and can be bounded by an integrable function.

\section{Details of proofs}
\label{proofinid}
In this section, we provide detailed proofs for the key results used in the previous sections. We start with the proof of proposition \ref{numer}.

\begin{proof}[Proof of proposition \ref{numer}]
Let $\theta_1 \in U^{c}$.  Since assumption \ref{KLind} holds, $E\log\frac{f_{\theta^*}(Y)}{f_{\theta_1}(Y)}> \delta$. By lemma 6.3 of \cite{Kleijn_van2006}, we have $\lim_{\alpha\downarrow 0} \frac{1-E\left( \frac{f_{\theta_1}}{f_{\theta^*}}\right)^{\alpha}}{\alpha}\geq \delta$. Therefore, $\exists \ \alpha' \in (0,1)$ such that $\frac{1-E\left( \frac{f_{\theta_1}}{f_{\theta^*}}\right)^{\alpha'}}{\alpha'}> \frac{\delta}{2}$ and hence: 
\begin{equation}
\label{step1}
E\left( \frac{f_{\theta_1}}{f_{\theta^*}}\right)^{\alpha'}<1-\alpha'\frac{\delta}{2} < e^{-\alpha'\frac{\delta}{2}}.
\end{equation}
 Now, define $A_{\theta_1}:= \left\{\theta \in \Theta : \ E\left[\frac{f_{\theta}}{f_{\theta_1}} \right]< e^{\frac{\delta}{2}} \right\}$. This set clearly contains $\theta_1$ and is an open set (by continuity of $E\left[\frac{f_{\theta}}{f_{\theta_1}}\right]$ as per assumption \ref{cnty}).  Let $\alpha=\frac{\alpha'}{2}$ and $\nu(\cdot)$ be any probability measure on $A_{\theta_1}$. Part (a) is established by the following inequality:
\begin{eqnarray*}
&& E\left[ \left(\int_{A_{\theta_1}}\frac{f_{\theta}}{f_{\theta^*}} d\nu(\theta)\right)^{\alpha}\right]=E\left[ \left(\frac{f_{\theta_1}}{f_{\theta^*}} \right)^{\alpha}\cdot \left(\int_{A_{\theta_1}}\frac{f_{\theta}}{f_{\theta_1}} d\nu(\theta)\right)^{\alpha}\right]\\
&&  \mbox{ (by Cauchy-Schwartz inequality)}\\
&&\leq \left(E\left[ \left(\frac{f_{\theta_1}}{f_{\theta^*}}\right)^{2\alpha}\right]\right)^{\frac{1}{2}} \cdot \left(E\left[ \left(\int_{A_{\theta_1}}\frac{f_{\theta}}{f_{\theta_1}} d\nu(\theta)\right)^{2\alpha}\right]\right)^{\frac{1}{2}}\\
&& \mbox{ (by Jensen's inequality on 2nd term)}\\
&&\leq \left(E\left[ \left(\frac{f_{\theta_1}}{f_{\theta^*}}\right)^{2\alpha}\right]\right)^{\frac{1}{2}} \cdot \left(\int_{A_{\theta_1}}E\left[ \frac{f_{\theta}}{f_{\theta_1}}\right]d\nu(\theta)\right)^{\alpha} \\ 
&&< e^{-\alpha\frac{\delta}{2}} \mbox{     (by equation \ref{step1} and definition of $A_{\theta_1}$)}.
\end{eqnarray*}
Part (b) can be shown using induction on $n$. Note that part (a) corresponds to $n=1$. Assume that the result holds for $n=k$. Then,
\begin{eqnarray*}
&&E\left[ \left(\int_{A_{\theta_1}} \prod_{i=1}^{k+1} \frac{f_{\theta}(y_i)}{f_{\theta^*}(y_i)}d\nu(\theta)\right)^{\alpha}\right]\\
&&=E\left[ \left. E\left[ \left(\int_{A_{\theta_1}}\frac{f_{\theta}(y_{k+1})}{f_{\theta^*}(y_{k+1})} d\nu_{y_1,y_2, \cdots, y_k}(\theta)\right)^{\alpha}  \right\vert Y_1,Y_2,\cdots, Y_k\right]  \cdot \left(\int_{A_{\theta_1}} \prod_{i=1}^{k} \frac{f_{\theta}(y_i)}{f_{\theta^*}(y_i)}d\nu(\theta)\right)^{\alpha}\right]\\
&&\mbox{where }  d\nu_{y_1,y_2, \cdots, y_k}(\theta)= \frac{\prod_{i=1}^{k} \frac{f_{\theta}(y_{i})}{f_{\theta^*}(y_{i})}d\nu(\theta)}{\int_{A_{\theta_1}}\prod_{i=1}^{k} \frac{f_{\theta}(y_{i})}{f_{\theta^*}(y_{i})}d\nu(\theta)}.
\end{eqnarray*}
Now, the result is obtained by applying part (a) on the first conditional expectation term and the induction hypothesis for n=k on the second term.
\end{proof}

We  now provide the details of the steps leading to propositions \ref{denominid} and \ref{numerinid}.  Similar to the i.i.d case, the first proposition is to handle the denominator $R'_{2n}$ of the posterior probability.

\begin{proof}[Proof of proposition \ref{denominid}.]~\\
From part (a) of assumption \ref{cntyinid}, it follows that the  collection $\left\{ E_{\bf x}\log\frac{f_{\theta^{*}({\bf x}_i)}}{f_{\theta({\bf x}_i)}}, \ i\geq 1\right\} $ is equi-continuous w.r.t. $\theta \in \Theta$. Part (b) implies that $\left\{ E_{\bf x}\log^2\frac{f_{\theta^{*}({\bf x}_i)}}{f_{\theta({\bf x}_i)}}, \ i\geq 1\right\}$ is uniformly bounded. Hence, $\exists \ \delta\in(0,1)$ such that 
 \begin{eqnarray*}
 &&\left\{\sup_{{\bf x}\in \mathcal{X}}|\theta({\bf x})-\theta_1({\bf x})|<\delta \right\}\\
 && \subseteq V_{\epsilon}= \left\{ \theta : \sup_{i\geq 1} E_{{\bf x}_i}\log\frac{f_{\theta^*_{{\bf x}_i}}(Y_i)}{f_{\theta_{{\bf x}_i}}(Y_i)}<\epsilon, \ \sum_{i=1}^{\infty} \frac{1}{i^2}E_{{\bf x}_i}\left(\log\frac{f_{\theta^*_{{\bf x}_i}}(Y_i)}{f_{\theta_{{\bf x}_i}}(Y_i)}\right)^2<\infty\right\} 
 \end{eqnarray*}
 Assumption \ref{existenceposinid} will therefore ensure that the prior gives positive mass for the set $V_{\epsilon}$. Now, observing that $R'_{2n}\geq \int_{V_{\epsilon}}e^{\sum_{i=1}^{n}\log\left(\frac{f_{\theta_{{\bf x}_i}}(Y_i)}{f_{\theta^*_{{\bf x}_i}}(Y_i)} \right)}d\Pi(\theta) $ and an application of strong law of large numbers for independent random variables  leads to  \[\sum_{i=1}^{n}\log\left(\frac{f_{\theta_{{\bf x}_i}}(Y_i)}{f_{\theta^*_{{\bf x}_i}}(Y_i)} \right)>-2n\epsilon \ \ a.s. \]
Rest of the proof is in the lines of Lemma 4.4.1 of \cite{Ghosh_RVR2003}.
\end{proof}

The next two lemmas help in proving proposition \ref{numerinid}. The  lemma below essentially formalizes the fact that if the functions $\theta$ and $\theta_0$ differ at a point ${\bf x}_0$, then they will necessarily differ on a neighborhood around ${\bf x}_0$ as well.
 \begin{lemma}
 \label{lemnhbd}
 Suppose assumption \ref{compactinid} holds. Let $\theta' \in U^c$ and ${\bf x}_0 \in \mathcal{X}$ be such that $|\theta'_{{\bf x}_0}-\theta^*_{{\bf x}_0}|>\epsilon$. Then  $\exists \ \delta'$ such that $\forall \ {\bf x} \ : \ \|{\bf x}-{\bf x}_0\|<\delta'$ we have $|\theta'_{\bf x}-\theta^*_{\bf x}| \geq \frac{\epsilon}{2}$.
 \begin{proof}
 
 Since assumption \ref{compactinid} holds, by Arzela-Ascoli theorem, we have the following:
 \begin{itemize}
 \item[{(i)}]  $\Theta$ is uniformly bounded, i.e. $\exists \ M$ such that $|\theta({\bf x})|\leq M$ $\ \forall \ \theta\in \Theta$ and ${\bf x}\in \mathcal{X}$.
 \item[{(ii)}]$\Theta$ is equi-uniformly-continuous, i.e. for ${\bf x}_0\in\mathcal{X}$, given $\epsilon>0$, $\exists \ \delta>0$ such that $\forall \ {\bf x} : \ \|{\bf x}-{\bf x}_0\|<\delta$, $|\theta_{\bf x}-\theta_{{\bf x}_0}|$ $<\epsilon$,  $\ \forall \ \theta\in\Theta$.
 \end{itemize}
 
 Without loss of generality, for $\theta'\in U^c$, we have $\theta'_{{\bf x}_0}-\theta^*_{{\bf x}_0}>\epsilon$. By (ii) above,i.e. equicontinuity, $\exists \ \delta'$ such that $\forall \ \|{\bf x}-{\bf x}_0\|<\delta'$, we have $|\theta_{\bf x}-\theta_{{\bf x}_0}|<\frac{\epsilon}{4}$, $\forall \ \theta\in\Theta$ . In particular, for such ${\bf x}$,  $|\theta^*_{{\bf x}}-\theta^*_{{\bf x}_0}|<\frac{\epsilon}{4}$.  Therefore,
 \begin{eqnarray*}
  \theta'_{\bf x}-\theta^*_{\bf x} = \theta'_{\bf x}-\theta'_{{\bf x}_0} + \theta'_{{\bf x}_0}-\theta^*_{{\bf x}_0} +\theta^*_{{\bf x}_0}-\theta^*_{{\bf x}}\geq  -\frac{\epsilon}{4} + \epsilon - \frac{\epsilon}{4} = \frac{\epsilon}{2}.
 \end{eqnarray*}
 \end{proof}
 \end{lemma}

 \begin{lemma}
 \label{sepcondinid}
 Let $U^{c}=\left\{\theta: \ \sup_{{\bf x} \in \mathcal{X}}|\theta({\bf x})-\theta^*({\bf x})|>\epsilon \right\}$.  If  assumptions \ref{existenceposinid} to \ref{denseinid} hold, then  $\exists$  $\delta_1\in (0,1) $  such that for every $\theta'\in U^{c}$, an $ \alpha'\in (0,1)$ can be chosen such that 
 $E  \left(\prod_{i=1}^{n}\frac{f_{\theta'_{{\bf x}_i}}(Y_i)}{f_{\theta^*_{{\bf x}_i}}(Y_i)} \right)^{\alpha'} < e^{-n\alpha' \delta}\ $ for all sufficiently large $n$.
 \begin{proof}
  For $\theta'\in U^{c}$, let ${\bf x}_0$ be such that $|\theta'({\bf x}_0)-\theta^*({\bf x}_0)|>\epsilon$. Then by lemma \ref{lemnhbd}, $\exists \ \delta'$ such that $\forall \ {\bf x} \in A_{{\bf x}_0,\delta'}:= \{ {\bf x} \ : \ \|{\bf x}-{\bf x}_0\|<\delta'\}$, we have $|\theta'_{\bf x}-\theta^*_{\bf x}| \geq \frac{\epsilon}{2}$. Therefore, by assumption \ref{KLindinid}, $\exists \ \delta \in (0,1)$ such that $E_{{\bf x}}\log\frac{f_{\theta^*_{\bf x}}}{f_{\theta'_{{\bf x}}}} \geq \delta$ for all  ${\bf x} \in A_{{\bf x}_0, \delta'}$.  
  
For $({\bf x},t,t') \in\mathcal{X}\times [-M,M]^2$, let $ g_{\alpha}({\bf x},t,t'):= \frac{1-E_{\bf x}\left( \frac{f_{t}}{f_{t'}}\right)^{\alpha}}{\alpha}$. By Lemma 6.3 of \cite{Kleijn_van2006}, we have that, $g_{\alpha}({\bf x},t,t') $ increases to $E_{\bf x}\log\frac{f_{t'}}{f_{t}} $ as $\alpha \downarrow 0$. By assumption \ref{cntyinid}, both $g_\alpha(\cdot,\cdot,\cdot)$ and the limiting function are continuous in $({\bf x},t,t')$, which is in the compact set $\mathcal{X}\times [-M,M]^2$. Hence, it follows by Dini's theorem that this convergence is uniform. i.e.,
  \[ \lim_{\alpha\downarrow 0} \frac{1-E_{\bf x}\left( \frac{f_{t}}{f_{t'}}\right)^{\alpha}}{\alpha}\uparrow E_{\bf x}\log\frac{f_{t'}}{f_{t}} \ \mbox{uniformly on } \mathcal{X}\times [-M,M]^2.  \] 
   Let  $\kappa:=\kappa({\bf x}_0,\delta')$ as in the assumption \ref{denseinid}. Then,   $\exists \ 0<\alpha'<1$ such that $g_{\alpha'}({\bf x},t,t')> E_{\bf x}\log\frac{f_{t'}}{f_{t}}- \kappa\frac{\delta}{2}, \forall \ ({\bf x},t,t')\in \  \mathcal{X}\times [-M,M]^2$. In particular, $g_{\alpha'}({\bf x}_i,\theta_{{\bf x}_i},\theta^*_{{\bf x}_i})\geq E_{{\bf x}_i}\log\frac{f_{\theta^*_{{\bf x}_i}}}{f_{\theta_{{\bf x}_i}}}-\kappa\frac{\delta}{2}$ $\forall \ i\geq \ 1, \theta\in\Theta$.  Also, in general $E_{{\bf x}_i}\log\frac{f_{\theta^*_{{\bf x}_i}}}{f_{\theta_{{\bf x}_i}}}\geq 0$. Combining this with the observation we made at the beginning of the proof that $E_{{\bf x}}\log\frac{f_{\theta^*_{\bf x}}}{f_{\theta'_{{\bf x}}}} \geq \delta$   $ \forall \ x\in A_{x_0, \delta'}$, we get:
    \begin{equation}
    \label{stepp}
    g_{\alpha'}({\bf x}_i,\theta'_{{\bf x}_i},\theta^*_{{\bf x}_i})\geq \delta\cdot I_{A_{{\bf x}_0,\delta'}}({\bf x}_i)-\kappa\frac{\delta}{2},
    \end{equation}
where $I_{A_{{\bf x}_0,\delta'}}({\bf x})$ is the indicator function which is $1$ when ${\bf x}\in A_{{\bf x}_0, \delta'}$ and $0$ otherwise. Note that, by assumption \ref{denseinid}, for sufficiently large n, $\frac{1}{n}\sum_{i=1}^{n}I_{A_{{\bf x}_0, \delta'}}({\bf x}_i) > \frac{3\kappa}{4}$. Using this along with a bit of algebra on equation (\ref{stepp}), we can conclude that the following inequality holds  for sufficiently large $n$:
  \begin{eqnarray*}
   E\left(\prod_{i=1}^{n} \frac{f_{\theta'_{{\bf x}_i}}(Y_i)}{f_{\theta^*_{{\bf x}_i}}(Y_i)}\right)^{\alpha'} && \leq e^{-\delta\sum_{i=1}^{n}\cdot I_{A_{{\bf x}_0,\delta'}}({\bf x}_i)+n\kappa\frac{\delta}{2}. } \leq e^{-n\kappa\frac{\delta}{4}}.
   \end{eqnarray*}
 The result follows by assigning $\delta_1:= \kappa\frac{\delta}{4}$.
 \end{proof}
 \end{lemma}
 
 \begin{proof}[Proof of proposition \ref{numerinid}.]~\\
  First, we claim by assumption \ref{cntyinid} that the collection of functions  $\left\{ E_{\bf x}\frac{f_{\theta({\bf x}_i)}}{f_{\theta'({\bf x}_i)}}, \ i\geq 1\right\} $ is equi-continuous w.r.t the sup-norm metric on $\Theta$. 
 Note that $E_{\bf x}\frac{f_t}{f_{t'}}$ is a  continuous function on a compact set $\mathcal{X}\times [-M,M]^2$. Hence, it is uniformly continuous. So, given $\epsilon>0$, $\exists \  \delta$ such that if $\|{\bf x}-{\bf x}_1\|<\delta$ , $|t-t_1|<\delta$ and $|t'-t'_1|<\delta$ then  $\left\vert E_{{\bf x}_1}\frac{f_{t_1}}{f_{t'_1}}-E_{\bf x}\frac{f_t}{f_{t'}} \right\vert <\epsilon$. In particular, let $\theta, \theta_1 \in \Theta$ be such that $\sup_{{\bf x}\in \mathcal{X}}|\theta({\bf x})-\theta_1({\bf x})|<\delta$. Then for any ${\bf x} \in \mathcal{X}$, taking ${\bf x}_1={\bf x}$, $t'=t'_1=\theta({\bf x})$ and $t=\theta({\bf x})$, $t_1=\theta_1({\bf x})$, we get  $\left\vert E_{\bf x}\frac{f_{\theta({\bf x})}}{f_{\theta'({\bf x})}}- E_{\bf x}\frac{f_{\theta_1({\bf x})}}{f_{\theta'({\bf x})}}\right\vert <\epsilon$. Hence the collection of functions  $\left\{ E_{\bf x}\frac{f_{\theta({\bf x}_i)}}{f_{\theta'({\bf x}_i)}}, \ i\geq 1\right\} $ is equicontinuous in $\theta$ w.r.t supnorm metric.
 
  Define $A_{\theta'}:= \left\{\theta \in \Theta : \ E_{{\bf x}_i}\left[\frac{f_{\theta_{{\bf x}_i}}}{f_{\theta'_{{\bf x}_i}}} \right]< e^{\frac{\delta}{2}} \ ,\forall i\geq 1\right\}$. This set clearly contains $\theta'$ and it is an open set due to  equi-continuity. By lemma \ref{sepcondinid},  $\exists \ \alpha'\in (0,1)$  such that
  \[  E  \left(\prod_{i=1}^{n}\frac{f_{\theta'_{{\bf x}_i}}(Y_i)}{f_{\theta^*_{{\bf x}_i}}(Y_i)} \right)^{\alpha'} < e^{-n\alpha' \delta}\ \ \mbox{ for all sufficiently large } n. \]  
   Let $\alpha=\alpha'/2$. Then, for sufficiently large $n$, 
  \begin{eqnarray*}
 &&E\left[ \left(\int_{A_{\theta'}}\prod_{i=1}^{n}\frac{f_{\theta_{{\bf x}_i}}(Y_i)}{f_{\theta^*_{{\bf x}_i}}(Y_i)} d\nu(\theta) \right)^{\alpha}\right]\\
&&= E\left[\left( \frac{f_{\theta'_{{\bf x}_i}}(Y_i)}{f_{\theta^*_{{\bf x}_i}}(Y_i)}  \right)^{\alpha} \left(\int_{A_{\theta'}}\prod_{i=1}^{n}\frac{f_{\theta_{{\bf x}_i}}(Y_i)}{f_{\theta'_{{\bf x}_i}}(Y_i)} d\nu(\theta) \right)^{\alpha}\right] \\
&& \mbox{(By Cauchy-Schwartz inequality)}\\
&&\leq \left(E\left[\left( \frac{f_{\theta'_{{\bf x}_i}}(Y_i)}{f_{\theta^*_{{\bf x}_i}}(Y_i)}  \right)^{2\alpha} \right]\right)^{\frac{1}{2}} \cdot \left(E\left[\left(\int_{A_{\theta'}}\prod_{i=1}^{n}\frac{f_{\theta_{{\bf x}_i}}(Y_i)}{f_{\theta'_{{\bf x}_i}}(Y_i)} d\nu(\theta) \right)^{2\alpha}\right]\right)^{\frac{1}{2}} \\
&&\mbox{(By Jensen's inequality)}\\
&&\leq \left(E\left[\left( \frac{f_{\theta'_{{\bf x}_i}}(Y_i)}{f_{\theta^*_{{\bf x}_i}}(Y_i)}  \right)^{\alpha'} \right]\right)^{\frac{1}{2}} \cdot \left(\int_{A_{\theta'}}E\left[\prod_{i=1}^{n}\frac{f_{\theta_{{\bf x}_i}}(Y_i)}{f_{\theta'_{{\bf x}_i}}(Y_i)}\right] d\nu(\theta) \right)^{\frac{\alpha'}{2}} \\
&& < e^{-n\alpha'\frac{\delta}{2}} \cdot e^{n\alpha'\frac{\delta}{4}} = e^{-n\alpha\frac{\delta}{2}}.
 \end{eqnarray*}
 \end{proof}
 
\section*{Acknowledgements}
This work was supported by a research grant provided by Indian Institute of Management Ahmedabad, India.

\end{document}